\documentclass[12pt,leqno]{amsart}
\usepackage{amssymb}
\usepackage{geometry}
\usepackage [english]{babel}
\usepackage [autostyle, english = american]{csquotes}
\MakeOuterQuote{"}
\begin{document}
	\theoremstyle{plain}
	\newtheorem{thm}{Theorem}[section]
	\newtheorem*{thm1}{Theorem 1}
	\newtheorem*{thm2}{Theorem 2}
	\newtheorem{lemma}[thm]{Lemma}
	\newtheorem{lem}[thm]{Lemma}
	\newtheorem{cor}[thm]{Corollary}
	\newtheorem{prop}[thm]{Proposition}
	\newtheorem{propose}[thm]{Proposition}
	\newtheorem{variant}[thm]{Variant}
	\newtheorem{conjecture}[thm]{Conjecture}
	\theoremstyle{definition}
	\newtheorem{notations}[thm]{Notations}
	\newtheorem{corollary}[thm]{Corollary}
	\newtheorem{notation}[thm]{Notation}
	\newtheorem{rem}[thm]{Remark}  
	\newtheorem{rmk}[thm]{Remark}
	\newtheorem{rmks}[thm]{Remarks}
	\newtheorem{defn}[thm]{Definition}
	\newtheorem{ex}[thm]{Example}
	\newtheorem{claim}[thm]{Claim}
	\newtheorem{Ackn}[thm]{Acknowledgement} 
	\numberwithin{equation}{section}
	\newcounter{elno}
	\newcommand{\iu}{{i\mkern1mu}}
	\newcommand{\rank}{{\rm rank}} \newcommand{\Ker}{{\rm Ker \ }}
	\newcommand{\Pic}{{\rm Pic}} \newcommand{\Div}{{\rm Div}}
	\newcommand{\Hom}{{\rm Hom}} \newcommand{\im}{{\rm im}}
	\newcommand{\Spec}{{\rm Spec \,}} \newcommand{\Sing}{{\rm Sing}}
	\newcommand{\sing}{{\rm sing}} \newcommand{\reg}{{\rm reg}}
	\newcommand{\Char}{{\rm char}} \newcommand{\Tr}{{\rm Tr}}
	\newcommand{\Gal}{{\rm Gal}} \newcommand{\Min}{{\rm Min \ }}
	\newcommand{\Max}{{\rm Max \ }} \newcommand{\Alb}{{\rm Alb}\,}
	\newcommand{\GL}{{\rm GL}\,} 
	\newcommand{\ie}{{\it i.e.\/},\ } \newcommand{\niso}{\not\cong}
	\newcommand{\nin}{\not\in}
	\newcommand{\soplus}[1]{\stackrel{#1}{\oplus}}
	\newcommand{\by}[1]{\stackrel{#1}{\rightarrow}}
	
	\newcommand{\longby}[1]{\stackrel{#1}{\longrightarrow}}
	\newcommand{\vlongby}[1]{\stackrel{#1}{\mbox{\large{$\longrightarrow$}}}}
	\newcommand{\ldownarrow}{\mbox{\Large{\Large{$\downarrow$}}}}
	\newcommand{\lsearrow}{\mbox{\Large{$\searrow$}}}
	\renewcommand{\d}{\stackrel{\mbox{\scriptsize{$\bullet$}}}{}}
	\newcommand{\dlog}{{\rm dlog}\,} 
	\newcommand{\longto}{\longrightarrow}
	\newcommand{\vlongto}{\mbox{{\Large{$\longto$}}}}
	\newcommand{\limdir}[1]{{\displaystyle{\mathop{\rm lim}_{\buildrel\longrightarrow\over{#1}}}}\,}
	\newcommand{\liminv}[1]{{\displaystyle{\mathop{\rm lim}_{\buildrel\longleftarrow\over{#1}}}}\,}
	\newcommand{\norm}[1]{\mbox{$\parallel{#1}\parallel$}}
		\newcommand{\into}{\hookrightarrow} \newcommand{\image}{{\rm image}\,}
		\newcommand{\Lie}{{\rm Lie}\,} 
		\newcommand{\CM}{\rm CM}
		\newcommand{\sext}{\mbox{${\mathcal E}xt\,$}} 
		\newcommand{\shom}{\mbox{${\mathcal H}om\,$}} 
		\newcommand{\coker}{{\rm coker}\,} 
		\newcommand{\sm}{{\rm sm}}
		\newcommand{\tensor}{\otimes}
		\renewcommand{\iff}{\mbox{ $\Longleftrightarrow$ }}
		\newcommand{\supp}{{\rm supp}\,}
		\newcommand{\ext}[1]{\stackrel{#1}{\wedge}}
		\newcommand{\onto}{\mbox{$\,\>>>\hspace{-.5cm}\to\hspace{.15cm}$}}
		\newcommand{\propsubset} {\mbox{$\textstyle{
					\subseteq_{\kern-5pt\raise-1pt\hbox{\mbox{\tiny{$/$}}}}}$}}
		\newcommand{\sA}{{\mathcal A}}
		\newcommand{\sB}{{\mathcal B}} \newcommand{\sC}{{\mathcal C}}
		\newcommand{\sD}{{\mathcal D}} \newcommand{\sE}{{\mathcal E}}
		\newcommand{\sF}{{\mathcal F}} \newcommand{\sG}{{\mathcal G}}
		\newcommand{\sH}{{\mathcal H}} \newcommand{\sI}{{\mathcal I}}
		\newcommand{\sJ}{{\mathcal J}} \newcommand{\sK}{{\mathcal K}}
		\newcommand{\sL}{{\mathcal L}} \newcommand{\sM}{{\mathcal M}}
		
		\newcommand{\sN}{{\mathcal N}} \newcommand{\sO}{{\mathcal O}}
		\newcommand{\sP}{{\mathcal P}} \newcommand{\sQ}{{\mathcal Q}}
		\newcommand{\sR}{{\mathcal R}} \newcommand{\sS}{{\mathcal S}}
		\newcommand{\sT}{{\mathcal T}} \newcommand{\sU}{{\mathcal U}}
		\newcommand{\sV}{{\mathcal V}} \newcommand{\sW}{{\mathcal W}}
		\newcommand{\sX}{{\mathcal X}} \newcommand{\sY}{{\mathcal Y}}
		\newcommand{\sZ}{{\mathcal Z}} \newcommand{\ccL}{\sL}
		\newcommand{\A}{{\mathbb A}} \newcommand{\B}{{\mathbb
				B}} \newcommand{\C}{{\mathbb C}} \newcommand{\D}{{\mathbb D}}
		\newcommand{\E}{{\mathbb E}} \newcommand{\F}{{\mathbb F}}
		\newcommand{\G}{{\mathbb G}} \newcommand{\HH}{{\mathbb H}}
		\newcommand{\I}{{\mathbb I}} \newcommand{\J}{{\mathbb J}}
		\newcommand{\M}{{\mathbb M}} \newcommand{\N}{{\mathbb N}}
		\renewcommand{\P}{{\mathbb P}} \newcommand{\Q}{{\mathbb Q}}
		
		\newcommand{\R}{{\mathbb R}} \newcommand{\T}{{\mathbb T}}
		\newcommand{\Z}{{\mathbb Z}}

		\title[HK density function of tensor product and Fourier transformation]{Hilbert-Kunz Density function of tensor product and Fourier transformation}

		\author[M. Mondal]{Mandira Mondal}
		\date{}
		\address{}
		\email{}

		\maketitle
		
		\begin{abstract}
			For a standard graded ring $R$ of dimension $\geq 2$ over a perfect field of characteristic $p>0$ and a homogeneous ideal $I$ of finite colength, the HK density function of $R$ with respect to $I$ is a compactly supported continuous function $f_{R, I}:[0, \infty)\longto [0, \infty)$, whose integration yields the \mbox{HK} multiplicity $e_{HK}(R, I)$.
			
			Here we answer a question of V. Trivedi about the Hilbert-Kunz density function of the tensor product of standard graded rings and show that it is the convolution of the Hilbert-Kunz density function of the factor rings.  Using Fourier transform, as a corollary we get \mbox{HK} multiplicity of the tensor product of rings is product of the HK multiplicity of the factor rings. We compute the Fourier transform of the \mbox{HK} density function of a projective curve. 
		\end{abstract}
		\section{Introduction}
		Let $R$ be a  Noetherian ring of prime characteristic
		$p >0$ and of dimension $d$ and let
		$I\subseteq R$ be an ideal of finite colength. The
		Hilbert-Kunz
		multiplicity of $R$ with
		respect to $I$ is defined as
		$$e_{HK}(R, I) = \lim_{n\to \infty}\frac{\ell_R(R/I^{[q]})}{q^{d}},$$
		where
		$q=p^n$, $I^{[q]}$ is the $n$-th Frobenius power of the ideal $I$, i.e., the ideal generated by $q$-th powers of the elements of $I$ and  $\ell_R(R/I^{[q]})$ denotes the
		length of
		the
		$R$-module $R/I^{[q]}$. Existence of the above limit was proved by
		Monsky \cite{Mo}. Hilbert-Kunz multiplicity has proved to be a more subtle invariant than
		the Hilbert-Samuel multiplicity and it often reflects the `characteristic $p$ property' of the
		ring. However, it has been difficult to compute as various 
		standard techniques used for studying Hilbert-Samuel multiplicities are not applicable 
		for $e_{HK}$.
		
		In order to study $e_{HK}$, 
		when  $R$ is a standard graded ring ($\dim~R\geq 2)$ over a perfect field
		of characteristic $p>0$ and $I$ is a homogeneous ideal of finite colength, Trivedi has introduced the notion of Hilbert-Kunz density function
		and showed its relation with the HK multiplicity \cite[Theorem 1.1]{Tri18}: 
		{\it the sequence of functions $\{f_n(R, I): [0,\infty) \longto 
			[0, \infty)\}_n$ given by 
			$$x\mapsto \frac{1}{q^{d-1}}\ell(R/I^{[q]})_{\lfloor qx\rfloor}$$ converges uniformly to a compactly supported continuous function 
			$f_{R, I}:[0,\infty) \longto 
			[0, \infty)$ such that 
			$$e_{HK}(R, I)=\int_0^\infty f_{R, I}(x) \ \mathrm{d}x.$$}
		\noindent The function $f_{R, I}$ is called the \mbox{HK} density function for the pair $(R, I)$.  Existence of a uniformly convergent sequence makes the density function a more refined invariant and a useful tool in the graded situation (compared to $e_{HK}$).  HK density function inherits many properties of the HK multiplicity.  In addition, this function is multiplicative on the Segre product of rings \cite[Proposition 1.2]{Tri18}, whereas the $e_{HK}$ may not display any such properties \cite[Example 3.7]{Tri18}. 
		In this paper we prove the following result about `multiplicative' property of  \mbox{HK} density function for tensor product of standard graded rings. Note that the \mbox{HK} density function can be considered as a continuous function on
		the real line by setting its value $0$ on $(-\infty, 0]$.
		\begin{thm}
			\label{thtensor}
			Let $\{R(i)\}_{i=1}^n$ be \textit{Noetherian} standard graded rings of dimension $\geq 2$ over a perfect field $\Bbbk$ of characteristic $p>0$. Let $I(i)$ be a homogeneous ideal of $R(i)$ of finite colength and let $I(1)\boxtimes \cdots \boxtimes I(n)$ be the homogeneous ideal $\sum_i R(1)\otimes\cdots R(i-1)\otimes I(i)\otimes R(i+1)\otimes\cdots R(n)$ of $\otimes_i R(i)$. Then
			$$f_{\otimes_i R(i), \boxtimes_i I(i)}=f_{R(1), I(1)}*\cdots*f_{R(n), I(n)},$$
			where for two integrable function $f, g:\R\longto \R$ we denote the convolution function of $f$ and $g$ by $f*g$.

		\end{thm}
		
		As a corollary,  using Fourier transform of an integrable function,  we get $$e_{HK}(\otimes_i R(i), \boxtimes_i I(i))=e_{HK}(R(1), I(1))\cdots e_{HK}(R(n), I(n)).$$ These results first appeared in the Ph.D. thesis \cite[Section 1.3]{M}
		of the author.

		In \cite{TT}, it is shown that there exists a ring homomorphism from $\Pi:\Z[C_\Bbbk]\longto H(\C)[X]$, where $C_\Bbbk$ is a monoid consisting of appropriately defined equivalent classes of standard graded pairs $(R, I)$ over the field $\Bbbk$ and $H(\C)$ denote the ring of holomorphic functions over $\C$. This homomorphism is induced from a map of monoids 
		$\Phi: (C_\Bbbk, \otimes)\longto H(\C)$ which sends a class of a graded pair $(R, I)$ to the Holomorphic Fourier transform of its HK density function $\widehat{f_{R,I}}$ when dimension of $R\geq 2$. Theorem \ref{thtensor} essentially states the `multiplicative' property of this morphism for graded pairs of dimension $\geq 2$. 
		
		In Section 3,  we compute Fourier transformation of  the \mbox{HK} density function $f_{R, I}$ for the  homogeneous coordinate ring $R$ of a projective curve with respect to a homogeneous ideal $I$ of finite colength.  This function was computed for the pair $(R, {\bf m})$ where ${\bf m}$ is the homogeneous maximal ideal of $R$ in \cite[Example 1.3.7]{M}.  
		One can directly check that the Fourier transform is a holomorphic function without using Plancherel theorem. Like the HK density function $f_{R, I}$, the Fourier transform also retains the information about the strong HN-data of the associated syzygy bundle. Using Theorem \ref{thtensor}, we can compute Fourier transform of HK density function for tensor product of projective curves.

		\section{Hilbert Kunz density function on tensor product}
		We recall the definition and few properties of convolution of two functions.
		\begin{defn}\label{dfnconv}
			Let $f:\R\longto \R$ and $g:\R\longto \R$ be two integrable functions, then the convolution of $f$ and $g$ is
			$$f*g:\R\longto \R\ \text{ given by }x\mapsto\int_{-\infty}^{\infty}f(y)g(x-y)dy.$$ 
		\end{defn}
		When $f$ and $g$ are integrable, the above definition always makes sense. Suppose $f, g, h:\R\longto \R$ are integrable functios.
		It is clear that $f*g=g*f$. Using this property and Fubini's theorem one gets $f*(g*h)=(f*g)*h$.

		\begin{lemma}\label{lemtensor}
			Let $R$ and $S$ be standard graded rings over a field $\Bbbk$ of characteristic $p>0$ and let $I\subset R$ and $J\subset S$ be homogeneous ideals of $R$ and $S$, respectively, of finite colength. For an integer $m\geq 0$,
			$$\ell\left(\frac{R\otimes S}{(I\otimes S + R\otimes J)^{[q]}}\right)_{m}=\sum\limits_{i=0}^{m} \ell(R/I^{[q]})_{i} \ell(S/J^{[q]})_{m-i}.
			$$
		\end{lemma}
		\begin{proof}Note that $(I\otimes S + R\otimes J)^{[q]}=I^{[q]}\otimes S +R\otimes J^{[q]}$. Hence
			for $0\leq m<q,$ we have
			\begin{eqnarray*}
				\ell\left(\frac{R\otimes S}{(I\otimes S + R\otimes J)^{[q]}}\right)_{m}&=&\ell((R\otimes S)_{m}) = \sum\limits_{i=0}^{m} \ell(R_{i}\otimes S_{m-i}) \\&=& \sum\limits_{i=0}^{m} \ell(R/I^{[q]})_{i}\ell (S/J^{[q]})_{m-i}.
			\end{eqnarray*}

			To prove the result for $m\geq q$, we first note that $$\ell(I^{[q]}\otimes S +R\otimes J^{[q]})_{m}=\ell(I^{[q]}\otimes S)_{m} + \ell(R\otimes J^{[q]})_{m} - \ell(I^{[q])}\otimes J^{[q]})_{m}$$
			
			since $(I^{[q]}\otimes S)\cap(R\otimes J^{[q]}) = I^{[q]}\otimes J^{[q]}$. Now for $m\geq q$, we have
			
			$$\ell\left(\frac{R\otimes S}{(I\otimes S + R\otimes J)^{[q]}}\right)_{m}=\ell\left((R\otimes S)_{m}\right)-\ell(I^{[q]}\otimes S +R\otimes J^{[q]})_{m}$$
			\begin{eqnarray*}
				&=&\ell(R\otimes S)_{m}-\ell(I^{[q]}\otimes S)_{m} - \ell(R\otimes J^{[q]})_{m} + \ell(I^{[q]}\otimes J^{[q]})_{m}\\
				&=&\ell(R\otimes S)_{m}-[\ell(R\otimes S)_{m}-\ell(R/I^{[q]}\otimes S)_{m}]-[\ell(R\otimes S)_{m}-\ell(R\otimes S/J^{[q]})_{m}]\\
				&& +\sum\limits_{i=0}^{m} \left(\ell(R_{i})-\ell(R/I^{[q]})_{i}\right)\times\left(\ell(S_{m-i}) - \ell(S/J^{[q]})_{m-i}\right)\\
				&=&\sum\limits_{i=0}^{m}\ell(R/I^{[q]})_{i}\ell(S/J^{[q]})_{m-i}.
			\end{eqnarray*}
			Hence the proof.
		\end{proof}

		\begin{proof}[Proof of Theorem \ref{thtensor}]

			Note that 
			$$(\boxtimes_{i=1}^{n-1}I(i))\otimes R(n)+(\otimes_{i=1}^{n-1}R(i))\otimes I(n)=\boxtimes_{i=1}^n I(i).$$ Hence by induction and because of the associativity property of the convolution, it is enough to
			prove the result for $n=2$. 
			
			Let $R$ and $S$ be standard graded rings over a field $\Bbbk$ of characteristic $p>0$ of dimension $d$ and $d'$ respectively ($d, d'\geq 2$), and let $I\subset R$ and $J\subset S$ be homogeneous ideals of $R$ and $S$, respectively, of finite colength.
			By Lemma \ref{lemtensor},
			\begin{eqnarray}
				\label{eqone}
				f_{n}(R\otimes S, I\boxtimes J)(x)=\frac{1}{q^{d+d'-1}}\sum\limits_{i=0}^{m} \ell(R/I^{[q]})_{i}\ell (S/J^{[q]})_{m-i}.
			\end{eqnarray} 
			For each $n\in\N$, consider the function $\phi_n:[0,\infty)\longto[0, \infty)$ given by 
			$$x\mapsto\int_0^{x+\frac{1}{q}} f_n(R, I)(y)f_n(S, J)(x+\frac{1}{q}-y)dy.$$
			
			We claim the limit  $\phi(x)=\lim\limits_{n\rightarrow\infty}\phi_{n}(x)$ exists and
			 that $\phi(x)=(f_{R, I}*f_{S, J}) (x)$ for all $x\in [0,\infty)$.
			We have 
			\begin{eqnarray*}
				\phi(x)&=&\lim\limits_{n\rightarrow\infty}\int_0^{x+\frac{1}{q}} f_n(R, I)(y)f_n(S, J)(x+\frac{1}{q}-y)dy\\
				&=&\lim\limits_{n\rightarrow\infty}\int_0^{x} f_n(R, I)(y)f_n(S, J)(x+\frac{1}{q}-y)dy \\&&+ \lim\limits_{n\rightarrow\infty}\int_x^{x+\frac{1}{q}} f_n(R, I)(y)f_n(S, J)(x+\frac{1}{q}-y)dy.
			\end{eqnarray*}
			
			\noindent Since $\{f_n(R, I)\}_n\ \textrm{and}\ \{f_n(S, J)\}_n$ converge uniformly to $f_{R, I}$ and $f_{S, J}$ respectively, we have
			\begin{eqnarray*} 
				&&\lim\limits_{n\rightarrow\infty}\int_0^{x} f_n(R, I)(y)f_n(S, J)(x+\frac{1}{q}-y)dy\\&=&\int_0^{x} \lim\limits_{n\rightarrow\infty} f_n(R, I)(y)f_n(S, J)(x+\frac{1}{q}-y)dy=\int_0^{x} f(y)g(x-y)dy. 
			\end{eqnarray*}
			Since each $f_n(R, I)$ is compactly supported continuous function and $\{f_n(R, I)\}$ converges to $f_{R, I}$ uniformly, we have $\{f_n(R, I)\}$ is uniformly bounded. Similarly $\{f_n(S, J)\}$ is uniformly bounded. Hence we can choose a positive constant $C$ such that $\arrowvert f_n(R, I)(x)\arrowvert\leq C$ and $\arrowvert f_n(S, J)(x)\arrowvert\leq C$ for all $n\in\N$ and for all $x\in [0,\infty)$. Hence
			
			\begin{eqnarray*}&&\big\arrowvert\lim\limits_{n\rightarrow\infty}\int_x^{x+\frac{1}{q}} f_n(R, I)(y)f_n(S, J)(x+\frac{1}{q}-y)dy\big\arrowvert
				\\&=&\lim\limits_{n\rightarrow\infty}\int_0^{\frac{1}{q}} f_n(R, I)(x+y)f_n(S, J)(-y+\frac{1}{q})dy
				~\leq~\lim\limits_{n\rightarrow\infty}\int_0^{\frac{1}{q}} C dy = 0.
			\end{eqnarray*}
			Hence,
			$$\phi(x)=\int_0^{x} f_{R, I}(y)f_{S, J}(x-y)dy =(f_{R, I}*f_{S, J})(x). $$
			Now take $x_0=\frac{m_0}{q_0}$, with $m\in \N$, $q_0=p^{n_0}$ and $n_0\in\N$. Let $q\geq q_0$. Write $m=\lfloor qx_0\rfloor=m_0q/q_0$.
			We have
			\begin{eqnarray*}\phi_{n}(x_0)&=&\sum\limits_{i=0}^{m} \int_{\frac{i}{q}}^{\frac{i+1}{q}} f_n(R, I)(y)f_n(S, J)(x+\frac{1}{q}-y)dy\\
				&=&\sum\limits_{i=0}^{m} \int_{\frac{i}{q}}^{\frac{i+1}{q}} \frac{1}{q^{d+d'-2}}\ell(R/I^{[q]})_{i}\ell(S/J^{[q]})_{m-i}dy\\
		&=&\sum\limits_{i=0}^{m} \frac{1}{q^{d+d'-1}}\ell(R/I^{[q]})_{i}\ell(S/J^{[q]})_{m-i}=f_n(R\otimes S, I*J)(x_0).
			\end{eqnarray*}

			Hence for all $x\in \Lambda=\{m'/q'\in \R\mid q=p^{n'}, m', n'\in \N\}$, we have
			$$f_{R\otimes S, I\boxtimes J}(x)=\lim_n f_n(R\otimes S, I\boxtimes J)(x)=\lim_n\phi_n(x)=\phi(x). $$
			Since the continuous functions $f_{R\otimes S, I\boxtimes J}$ and $\phi$ agree on the dense set $\Lambda\subset [0,\infty)$, we have $f_{R\otimes S, I\boxtimes J}=\phi=f_{R, I}*f_{S, J}$.
		\end{proof}
		\begin{defn}
			Let $f :\R\longto \R$ is an integrable function. Then the Holomorphic Fourier transform of $f$ is defined as $\hat{f}:\C\longto \C$ given by
			$$\hat{f}(\xi)=\int_{-\infty}^{\infty}f(x)\mathrm{e}^{- \mathrm{i} x\xi}\mathrm{d}x, \text{ for }\xi\in\C.$$ 
		\end{defn}
		\begin{rmk}If $f:\R\longto \R$ is a compactly supported function with $f\in \sL^2(-A, A)$ for some constant $A>0$, then by Plancherel theorem, $\hat{f}$ is a holomorphic function and $|\hat{f}(\xi)|\leq Ce^{A|\xi|}$ for all $\xi\in \C$ where 
			$C=\int_{-A}^Af(x)dx$. Conversely, if $F:\C\longto\C$ is a holomorphic function such that $|F(\xi)|\leq Ce^{A|\xi|}$ for some constants $A,C>0$, then there exists a function $f\in \sL^2(-A, A)$ such that $F=\hat{f}$ \cite[Theorem 19.3]{Rudin}. Note that if $f, g:\R\longto\R$ continuous function such that $F=\hat{f}=\hat{g}$, then $f=g$. Thus the HK density function can be recovered from its Holomorphic Fourier transformation.
		\end{rmk}
		
	
		\begin{corollary}Let $R(i)$,  $I(i)$  for $i=1,\ldots, n$ be as in Theorem \ref{thtensor}. Then
			$$e_{HK}(\otimes_i R(i), \boxtimes_i I(i))=\prod_ie_{HK}(R(i),  I(i)). $$
		\end{corollary}
		\begin{proof}
			Note that, the \mbox{HK} density function $f_{R, I}$ for a standard graded ring $R$ of dimension $\geq 2$ and a homogeneous ideal $I$ of finite colength can be considered as a function on the real line $\R$, extending by $0$ on $(-\infty, 0].$  
			Hence we have 
			$$e_{HK}(R, I)=\int_{-\infty}^{\infty}f_{R, I}(x)\mathrm{d}x=\widehat{{f}_{R, I}}(0).$$
			By induction on $n$, it is enough to prove for $n=2$. Let $R$ and $S$ be standard graded rings of dimension $\geq 2$ and let $I\subset R$ and $J\subset S$ be homogeneous ideals of $R$ and $S$, respectively, of finite colength. Then
			$$e_{HK}(R\otimes S, I\boxtimes J)=\widehat{f_{R,I}*f_{S,J}}(0)=\widehat{f_{R, I}}(0)\widehat{f_{S,J}}(0)=e_{HK}(R,I)e_{HK}(S,J).$$
			Hence the assertion is proved by induction on $n$.
		\end{proof}

		\section{Fourier transform of HK density function for projective curves}

		Let $X$ be a nonsingular projective curve over an algebraically closed field $K$ of characteristic $p>0$ and let $V$ be a vector bundle of rank $r$ on $X$. We recall that $\text{deg}(V)=\text{deg}(\wedge^rV)$ where $\wedge^rV$ denotes the determinant line bundle of the bundle $V$, and slope of $V$ is defined as, $\mu(V)=\text{deg}(V)/\text{rank}(V).$
		The vector bundle $V$ is called \textit{semistable} if,  for all subbundles $V'\hookrightarrow V$, we have $\mu(V')\leq \mu(V).$
		A vector bundle $V$ on $X$ is called \textit{strongly semistable} if $F^{n*}V$ is semistable, where $F^n: X\longto X$ is the $n$-th iterated absolute Frobenius morphism, for all $n\geq 0$.
		
		\begin{thm}[Lemma 1.3.7, \cite{HN1975}]\label{thhnf}
			Let $X$ be a smooth projective curve and let $V$ be a vector bundle on $X$. Then there exists a unique filtration $($called Harder-Narasimhan filtration$)$, by subbundles of $V,$
			\begin{eqnarray}\label{hnf}
				0=V_0\subset V_1\subset\cdots\subset V_l=V
			\end{eqnarray}
			such that
			\begin{enumerate}
				\item $V_1, V_2/V_1\ldots, V/V_{l-1}$ are semistable vector bundles and
				\item $\mu(V_1)>\mu(V_2/V_1)>\cdots>\mu(V/V_{l-1}).$
			\end{enumerate}
		\end{thm}
		The filtration in (\ref{hnf}) is called \textit{strong Harder-Narasimhan filtration} if, in addition, all the quotients $V_i/V_{i-1}$ are strongly semistable vector bundles. It is known that there exists $n_0>0$ such that $F^{n*}V$ has a strong Harder-Narasimhan filtration, for all $n\geq n_0$ \cite[Theorem 2.7]{Lan2004}. 
		\begin{defn}\label{hnfdefn}
			With notations as in Theorem \ref{thhnf}, the slopes $\mu_i(V):=\mu(V_i/V_{i-1}),$ $i=1,\ldots, l$ are called the \textit{Harder-Narasimhan slopes} of $V$. 
	
			Let $n\in\N$ be such that $F^{n*}V$ has strong Harder-Narasimhan filtration
			$$0=E_0\subset E_1\subset\cdots\subset E_l=F^{n*}V.$$
			Define 
			$a_i(V):=\frac{\mu(E_i/E_{i-1})}{p^n}=\frac{\mu_i(F^{n*}V)}{p^n}$ for $i=1,\ldots, l$
			to be the \textit{normalised Harder-Narasimhan slopes} of $V$. Define $r_i(V)=\text{rank} (E_i/E_{i-1}),\ \text{ for }\ i=1,\ldots, l.$
			We call the set
			$(\{a_1,\ldots, a_l\}, \{r_1,\ldots, r_l\})$
			the \textit{strong $HN$ data} for $V$. It is easy to see that $a_i(V)$ and $r_i(V)$ are well defined.
		\end{defn}

		\begin{defn}For a vector bundle $V$ on $X$ and an ample line bundle $\sO_X(1)$ on $X$ we define the \mbox{HK} density function of $V$ with respect to $\sO_X(1)$, 
			$$f_{V, \sO_X(1)}:[0,\infty)\longto[0,\infty)\text{ given by }x\mapsto\lim_{n\to\infty} \frac{1}{p^n}h^1(X, F^{n^*}(\lfloor(x-1)p^n\rfloor)).$$
			It is known that $f_{V, \sO_X(1)}$ is completely determined by the strong HN data of $V$ and $d=$ degree of the line bundle $\sO_X(1)$\cite[Section 6]{TW}.
		\end{defn}
		\begin{propose}\label{curveHKdFT}
			Let $R$ be a standard graded domain over a perfect field $\Bbbk$ of characteristic $p>0$ and let $I$ be a homogeneous ideal of $R$ of finite colength. Let ${h_1,\ldots,h_m}$ be a homogeneous set of generators of $I$ with $\deg h_i=e_i$ and $e_1=\cdots=e_{s_1}<e_{s_1+1}=\cdots=e_{s_1+s_2}<\cdots<e_{s_1+\dots+s_{l'-1}+1}=\cdots=e_{s_1+\cdots+s_{l'}}=e_m$ for some $s_1,\ldots, s_{l'}\geq 1$. Let $X$ be the normalization of the projective curve $\textnormal{\mbox{Proj}}~R$, $d=\deg(X)$ and let $\{(a_1,\ldots,a_{l}), (r_1,\ldots,r_l)\}$ be the strong \mbox{HN} data of the syzygy bundle $\sV$ given by the exact sequence
			$$0\longto \sV\longto \oplus_{i=1}^m\sO_X(1-e_i)\stackrel{(h_1, \ldots, h_m)}{\longto} \sO_X(1)\longto 0.$$
			Then 	\begin{eqnarray*}
				\widehat{{f}_{R, I}}(\xi)
				&=&-\frac{d}{\xi^2} \left[\sum_{j=1}^{l}r_j\mathrm{e}^{- \mathrm{i}(1-\frac{a_j}{d})\xi}-\sum_{k=1}^{l'}s_k\mathrm{e}^{-\mathrm{i}(1-\frac{1-e(k)}{d})\xi}+\textnormal{rank}(\sO_X(1))\mathrm{e}^{- \mathrm{i} (1-\frac{\deg \sO_X(1)}{d})\xi}\right],
			\end{eqnarray*}
			where $e(k)=e_{s_1+\cdots+s_k}$ for $k=1,\ldots,l'$.
		\end{propose}
		\begin{proof}
			
			Set $a_0=0$. Using the description of  
			$f_{\sV, \sO_X(1)}(x)$ from \cite[Section 6.1]{TW}, we have
			\begin{eqnarray}\label{eq1}
				\nonumber &&\int_1^{\infty}f_{\sV, \sO_X(1)}(x)\mathrm{e}^{-\mathrm{i}x\xi} \mathrm{d}x\\ 
				\nonumber &\quad & =\sum_{j=0}^{l}\int_{1-\frac{a_j}{d}}^{1-\frac{a_{j+1}}{d}}\sum_{k=j+1}^{l}-(a_kr_k+r_kd(x-1))\mathrm{e}^{-\mathrm{i}x\xi}\mathrm{d}x\\
				\nonumber & \quad &=\sum_{j=1}^{l}\int_{1}^{1-\frac{a_{j}}{d}}-(a_jr_j+r_jd(x-1))\mathrm{e}^{-\mathrm{i}x\xi}\mathrm{d}x\\
				\nonumber &\quad &=\sum_{j=1}^{l}(r_jd-a_jr_j)\int_{1}^{1-\frac{a_j}{d}}\mathrm{e}^{-\mathrm{i} x\xi} \mathrm{d}x-r_jd\int_{1}^{1-\frac{a_j}{d}}x\mathrm{e}^{-\mathrm{i} x\xi}\mathrm{d}x\\
				&\quad &=-\frac{d}{\xi^2} \sum_{j=1}^{l}r_j\mathrm{e}^{- \mathrm{i}(1-\frac{a_j}{d})\xi}+\frac{d}{\xi^2}(\sum_{j=1}^{l}r_j)\mathrm{e}^{-\mathrm{i}\xi}+\frac{\mathrm{i}}{\xi}(\sum_{j=1}^{l}a_jr_j)\mathrm{e}^{-\mathrm{i}\xi}.
			\end{eqnarray}
			
			Write $M:=\oplus_{i=1}^m\sO_X(1-e_i)$. Let us denote $e(k):=e_{s_0+\cdots+s_{k-1}+1}=\cdots=e_{s_0+\cdots+s_k}$ for $k=1,\ldots,l'$ with the convention $s_0=0$. Then the strong HN data for $M$ is $(\{1-e(1),\ldots,1-e(l')\},\{s_1,\ldots,s_{l'}\})$. 	Similarly as the vector bundle $\sV$, for the vector bundle $M$ we get 
			\begin{eqnarray}\label{eq2}
				\nonumber \int_1^{\infty}f_{M, \sO_X(1)}(x)\mathrm{e}^{-\mathrm{i}x\xi} \mathrm{d}x &=&-\frac{d}{\xi^2} \sum_{k=1}^{l'}s_k\mathrm{e}^{- \mathrm{i}(1-\frac{1-e(k)}{d})\xi}\\ 
				&\quad &+\frac{d}{\xi^2}(\sum_{k=1}^{l'}s_k)\mathrm{e}^{- \mathrm{i}\xi}+\frac{\mathrm{i}}{\xi}\big(\sum_{k=1}^{l'}[1-e(k)]s_k\big)\mathrm{e}^{-\mathrm{i}\xi}.
			\end{eqnarray}
			
			Also we have 
			\begin{eqnarray}\label{eq3}
				&&\int_0^{1}dx\mathrm{e}^{-\mathrm{i} x\xi} \mathrm{d}x =\quad =-\frac{d}{\xi^2}+\frac{d}{\xi^2}\mathrm{e}^{- \mathrm{i}\xi}+\frac{id}{\xi}\mathrm{e}^{-\mathrm{i}\xi}.
			\end{eqnarray}
			
			Since $f_{R, I}(x)=f_{\sV, \sO_X(1)}(x)-f_{M, \sO_X(1)}(x)$ for $x\in [0, \infty)$ \cite[Theorem 6.3]{TW}, we have
			\begin{eqnarray*}
				\widehat{{f}_{R, I}}(\xi)&=&\int_0^{\infty}f_{R, I}(x)\mathrm{e}^{-\mathrm{ i} x\xi}\mathrm{d}x
				\\&=&\int_0^{1}dx\mathrm{e}^{- \mathrm{i} x\xi} \mathrm{d}x+\int_{1}^{\infty}f_{\sV, \sO_X(1)}(x)\mathrm{e}^{-\mathrm{i}x\xi} \mathrm{d}x-\int_1^{\infty}f_{M, \sO_X(1)}(x)\mathrm{e}^{-\mathrm{i}x\xi} \mathrm{d}x.
			\end{eqnarray*}
			
			Since $\sum_j a_jr_j-\sum_k [1-e(k)]s_k=d$ and $\sum_k s_k=\sum_j r_j +1$, using (\ref{eq1}), (\ref{eq2}) and (\ref{eq3}) we get
			
			\begin{eqnarray*}
				\widehat{{f}_{R, I}}(\xi)
				&=&-\frac{d}{\xi^2} \left[\sum_{j=1}^{l}r_j\mathrm{e}^{-\mathrm{ i}(1-\frac{a_j}{d})\xi}-\sum_{k=1}^{l'}s_k\mathrm{e}^{-\mathrm{ i}(1-\frac{1-e(k)}{d})\xi}+1\right]\\
				&=&-\frac{d}{\xi^2} \left[\sum_{j=1}^{l}r_j\mathrm{e}^{-\mathrm{ i}(1-\frac{a_j}{d})\xi}-\sum_{k=1}^{l'}s_k\mathrm{e}^{- \mathrm{i}(1-\frac{1-e(k)}{d})\xi}+\text{rank}(\sO_X(1))\mathrm{e}^{-\mathrm{ i} (1-\frac{\deg \sO_X(1)}{d})\xi}\right].
			\end{eqnarray*}
		\end{proof}
		
		\begin{corollary}
			Adopt the notations of Proposition \ref{curveHKdFT}. Then $\widehat{f_{R, I}}$ is a holomprphic function.
		\end{corollary}		
		\begin{proof} Here we show that $\widehat{f_{R, I}}$ is a holomprphic function without using Plancherel theorem. To see this write the power series expansion of each term in $\widehat{f_{R, I}}$:
			\begin{eqnarray*}
				\widehat{{f}_{R, I}}(\xi)
				&=&-\frac{d}{\xi^2} \left[\sum_{j=1}^{l}r_j\mathrm{e}^{-\mathrm{ i}(1-\frac{a_j}{d})\xi}-\sum_{k=1}^{l'}s_k\mathrm{e}^{-\mathrm{ i}(1-\frac{1-e(k)}{d})\xi}+1\right]\\
				&=&-\frac{d}{\xi^2} \left[\sum_{j=1}^{l}r_j-\mathrm{i}\sum_{j=1}^{l}r_j(1-\frac{a_j}{d})\xi-\sum_{k=1}^{l'}s_k +\mathrm{i}\sum_{k=1}^{l'}s_k(1-\frac{1-e(k)}{d})\xi+1+h(\xi)\right]
			\end{eqnarray*}
			where $h:\C\longto \C$ is a holomorphic function with power series expansion involving terms $\xi^k$ for $k\geq 2$. Since
			$\sum_jr_j+1=\sum_k s_k$, we have
			\begin{eqnarray*}
				\widehat{{f}_{R, I}}(\xi)
				=\frac{\mathrm{i}d}{\xi}\left[-\sum_{j=1}^{l}r_j(1-\frac{a_j}{d}) +\sum_{k=1}^{l'}s_k(1-\frac{1-e(k)}{d})\right]-\frac{d}{\xi^2}h(\xi).
			\end{eqnarray*}
			Now since $\sum_j a_jr_j-\sum_k [1-e(k)]s_k=d$ and $\sum_k s_k=\sum_j r_j +1$, we have 
			$\widehat{{f}_{R, I}}(\xi)
			=-\frac{d}{\xi^2}h(\xi)$, which is a holomorphic function.
			
		\end{proof}
		
		\vspace{.2cm}
				
		\noindent{\bf Acknowledgement}\\
		
		The author thanks Prof.  V. Trivedi for posing the question about \mbox{HK} density function of tensor product of rings and also for suggesting possible application of Fourier transform to the theory of \mbox{HK} density function. The author is supported by DST grant DST/INSPIRE/04/2021/001549.

		\medskip
		
		\textsc{Department of Mathematics, Indian Institute of Technology Jodhpur, Rajasthan 342030, India.}
		\par\nopagebreak
		\noindent \textit{E-mail address}: \texttt{mandira@iitj.ac.in}
	\end{document}